\documentclass[12pt,a4paper]{amsart}
\usepackage{graphicx,amssymb}
\usepackage{amsmath,amssymb,amscd}
\input xy
\xyoption{all}
\usepackage[all]{xy}
\usepackage{hyperref}

\newcommand{\ra}{\rightarrow}

\newcommand{\PP}{\mathbb P}

\newcommand{\cO}{\mathcal{O}}

\newcommand{\Hom}{\mbox{Hom}}

\newcommand{\rk}{\mbox{rk}}

\newcommand{\Cl}{\operatorname{Cliff}}
\theoremstyle{plain}
\newtheorem{theorem}{Theorem}[section]
\newtheorem{lem}[theorem]{Lemma}
\newtheorem{prop}[theorem]{Proposition}
\newtheorem{cor}[theorem]{Corollary}

\newtheorem{rem}[theorem]{Remark}

\numberwithin{equation}{section}
\begin{document}
\title[Trigonal curves]{Bundles computing Clifford indices on trigonal curves}

\author{H. Lange}
\author{P. E. Newstead}

\address{H. Lange\\Department Mathematik\\
              Universit\"at Erlangen-N\"urnberg\\
              Cauerstrasse 11\\
              D-$91058$ Erlangen\\
              Germany}
              \email{lange@mi.uni-erlangen.de}
\address{P.E. Newstead\\Department of Mathematical Sciences\\
              University of Liverpool\\
              Peach Street, Liverpool L69 7ZL, UK}
\email{newstead@liv.ac.uk}

\date{\today}

\thanks{Both authors are members of the research group VBAC (Vector Bundles on Algebraic Curves). The second author 
would like to thank the Department Mathematik der Universit\"at 
         Erlangen-N\"urnberg for its hospitality}

\begin{abstract}
In this paper, we determine bundles which compute the higher Clifford indices for trigonal curves.
\end{abstract}
\maketitle

\section{Introduction}\label{intro}

In \cite{ln}, we introduced two definitions for Clifford indices of semistable vector  bundles on a smooth projective curve of  genus $g\ge4$. Many properties of these indices have been obtained in \cite{ln} and subsequent papers. It is an interesting question to determine the bundles which compute the Clifford indices; for results in ranks 2 and 3, see \cite{ln3, ln2}.

When $C$ has classical Clifford index $0$ (i.e. $C$ is hyperelliptic), it is already known that the higher Cifford indices are also $0$ and that all bundles computing them are direct sums of copies of $H^m$, where $H$ is the hyperelliptic line bundle (see \cite[Proposition 2]{re}). In this paper, we consider the next case, when $C$ is trigonal and therefore has classical Clifford index $1$. The higher Clifford indices are also equal to $1$, but one needs to determine the bundles that compute them. We show that all bundles which compute  the less restrictive Clifford index compute also the more restrictive one and determine many of these bundles. The answers are slightly different for $g=4$ (Theorems \ref{thm3.2} and \ref{thm3.4}) and $g\ge5$ (Theorem 4.7) (here our results are complete only for $g=5$, $g=6$ and $g\ge17$ due to the fact that in other cases we do not know whether a certain line bundle is normally generated).

In the final section, we consider trigonal curves of genus $3$. In this case, only one of the two higher Clifford indices is defined and it does not always take the value $1$. We determine all these indices and the bundles that compute them (Theorem \ref{thm5.3}).

We work throughout on a curve $C$ of genus $g\ge3$ defined over an algebraically closed field of characteristic $0$. For any vector bundle $E$ of rank $n$ on $C$, we write $d_E$ for the degree of $E$ and $\mu(E):=\frac{d_E}n$ for the slope of $E$.

\section{Preliminaries}\label{prelim}
We begin by recalling some definitions from \cite{ln}. For any curve $C$, the {\em gonality sequence} 
$d_1,d_2,\ldots,d_r,\ldots$ of $C$ is defined by 
$$
d_r := \min \{ d_L \;|\; L \; \mbox{a line bundle on} \; C \; \mbox{with} \; h^0(L) \geq r +1\}.
$$
The curve $C$ is said to be {\it trigonal} if $d_1=3$, in which case (see \cite[Remark 4.5(b)]{ln})
\begin{equation}\label{eq1.1}
d_r =\left\{ \begin{array}{lcl}
             3r &  &  1 \leq r \leq \left[ \frac{g-1}{3} \right],\\
             r + g - 1 - \left[ \frac{g-r-1}{2} \right] & \mbox{for} & \left[ \frac{g-1}{3} \right] < r \leq g-1,\\
             r+g & & r \geq g.
             \end{array}  \right. 
\end{equation}

For any vector bundle $E$ of rank $n$ on $C$, we define
$$
\gamma(E) := \frac{1}{n} \left(d_E - 2(h^0(E) -n)\right) = \mu(E) -2\frac{h^0(E)}{n} + 2.
$$
If $g \geq 4$, we then define, for any positive integer $n$,
$$
\Cl_n(C):= \min_{E} \left\{ \gamma(E) \;\left| 
\begin{array}{c} E \;\mbox{semistable of rank}\; n, \\
h^0(E) \geq 2n,\; \mu(E) \leq g-1
\end{array} \right. \right\}
$$
and
$$
\gamma_n(C):= \min_{E} \left\{ \gamma(E) \;\left| 
\begin{array}{c} E \;\mbox{semistable of rank}\; n, \\
h^0(E) \geq n+1,\; \mu(E) \leq g-1
\end{array} \right. \right\}.
$$
(In \cite{ln} and some other papers, $\Cl_n(C)$ was denoted by $\gamma'_n(C)$.) Note that $\Cl_1(C)=\gamma_1(C) = \Cl(C)$ is the usual Clifford index of the curve $C$. 
We say that $E$ {\it contributes to} $\Cl_n(C)$ (resp. $\gamma_n(C)$) if $E$ is semistable of rank $n$ with $h^0(E) \geq 2n$ (resp. $n+1$) and $\mu(E) \leq g-1$. 
If in addition $\gamma(E) = \Cl_n(C)$ (resp. $\gamma_n(C)$), we say that $E$ {\it computes} $\Cl_n(C)$ (resp. $\gamma_n(C)$). A trigonal curve of genus $g\ge4$ has $\Cl(C)=1$ and hence $\Cl_n(C)=\gamma_n(C)=1$ for all $n$ \cite[Proposition 2.6(a)]{ln}. We shall see later that, for $n\ge2$, $\gamma_n(C)$ can be defined even when $g=3$, but then it does not always take the value $1$.

\begin{lem} \label{lem2.1}
Let $C$ be a trigonal curve of genus $g \geq 4$ with trigonal bundle $T$ and let $E$ be a bundle of rank $n$ with $\gamma(E) = 1$. Then
$$
h^0(E \otimes T^*) + h^0(E^* \otimes T^* \otimes K_C) \geq n(g-3).
$$ 
\end{lem}

\begin{proof}
Since $\gamma(E) = 1$, we have
$$
d_E = n + 2(h^0(E) -n) = 2h^0(E) -n.
$$ 
Consider the evaluation sequence
$$
0 \ra T^* \ra H^0(T) \otimes \cO_C \ra T \ra 0.
$$ 
Tensoring with $E$, taking global sections and using Riemann-Roch, we obtain
\begin{eqnarray*}
2h^0(E) &\leq & h^0(E \otimes T^*) + h^0(E \otimes T)\\
& = & h^0(E \otimes T^*) + h^0(E^* \otimes T^* \otimes K_C) + 2h^0(E) + 2n - n(g-1).
\end{eqnarray*}
This implies the assertion.
\end{proof}

\begin{lem} \label{lem2.2}
Let $C$ be a trigonal curve of genus $g \geq 5$ with trigonal bundle $T$. Then $T$ is unique and there exists no non-trivial extension
\begin{equation*}
0 \ra T \ra  E \ra T \ra 0
\end{equation*}
with $h^0(E) = 4$. Moreover $T$ is the only bundle computing $\Cl(C)$.
\end{lem}

\begin{proof}
The uniqueness of $T$ is well known. We need to show that the map
$$
H^1(T^* \otimes T) \ra \Hom (H^0(T),H^1(T))
$$
is injective or equivalently 
$$
H^0(T) \otimes H^0(K_C \otimes T^*) \ra H^0(K_C)
$$
is surjective. Now $h^0(T) = 2, \; h^0(K_C \otimes T^*) = g-2$ and $h^0(K_C) = g$. Moreover, the kernel is isomorphic to $H^0(K_C \otimes T^{*2})$ which has dimension
$$
h^0(T^2) + 2g -8 + 1 -g = h^0(T^2) + g-7.
$$
We need to prove that $h^0(T^2) \leq 3$. This holds, because $d_3 \geq 7$ by \eqref{eq1.1}.

The last part follows from \eqref{eq1.1}.
 \end{proof}

\section{genus 4}\label{g4}

Let $C$ be a non-hyperelliptic curve of genus $g = 4$. So 
$$
\gamma_n(C) = \Cl_n(C) = \Cl(C) = 1
$$ 
for all $n$ and 
$$
d_1 = 3, \; d_2 = 5, \; d_3 = 6, \; d_4 = 8.
$$
We distinguish 2 cases.

\begin{enumerate}
 \item[(i)] There are 2 line bundles $M_1,\; M_2$ of degree 3 with $h^0 = 2$ and $K_C \simeq M_1 \otimes M_2$. 
\item[(ii)] There is a unique line bundle $M$ of degree 3 with $h^0(M) = 2$ and $K_C \simeq M^2$.
\end{enumerate}

\begin{lem} \label{lem3.1}
Let $C$ be a curve of type {\em (i)}. Then the map
$$
H^0(M_i) \otimes H^0(M_j) \ra H^0(M_i \otimes M_j)
$$
is surjective for $i,j \in\{1,2\}$. 
\end{lem}

\begin{proof}
Consider the evaluation sequence
$$
0 \ra M_i^* \ra H^0(M_i) \otimes \cO_C \ra M_i \ra 0.
$$ 
Tensoring with $M_j$ and taking global sections gives
$$
0 \ra H^0(M_i^* \otimes M_j) \ra H^0(M_i) \otimes H^0(M_j) \ra H^0(M_i \otimes M_j)
$$
If $i \neq j$, $h^0(M_i^* \otimes M_j) = 0$ and $h^0(M_i \otimes M_j) = h^0(K_C) = 4$. This implies the assertion.
If $i=j$, $h^0(M_i^* \otimes M_i) = 1$ and $h^0(M_i^2) = 3$ again giving surjectivity.
\end{proof}

\begin{theorem} \label{thm3.2}
Let $C$ be a curve of type {\em (i)} and $E$ be a bundle computing $\gamma_n(C)$. Then
$$
E \simeq \oplus_{i=1}^n M_{j_i}
$$
with $M_{j_i} = M_1$ or $M_2$. In particular $E$ computes $\Cl_n(C)$. 
\end{theorem}

\begin{proof}
We have $h^0(E) = n+s$ with $s \geq 1$. Since $\gamma_n(C) = 1$, it follows that $d_E = n + 2s$. On the other hand, by definition, $d_E \leq n(g-1) = 3n$. 
So $s \leq n$. 

By Lemma \ref{lem2.1} and the fact that $M_1 \otimes M_2 \simeq K_C$ we have
\begin{equation} \label{eq3.1}
h^0(E \otimes M_1^*) + h^0(E^* \otimes M_2) \geq n.
\end{equation}
Suppose first that there exists a non-zero homomorphism $M_1 \ra E$.
Then the semistability of $E$ implies that $M_1$ must be a subbundle of $E$ and $E$ must have degree $d_E = 3n$, i.e. $s=n$ and so $E$ computes $\Cl_n(C)$. 

If $n = 1$, this implies that $E \simeq M_1$. So suppose $n \geq 2$ and the theorem has been proved 
for all bundles of rank $n-1$. Then $E/M_1$ computes $\Cl_{n-1}(C)$ and by induction we have an exact sequence
$$
0 \ra M_1 \ra E \ra \oplus_{i=1}^{n-1} M_{j_i} \ra 0
$$
and all sections of $\oplus_{i=1}^{n-1} M_{j_i}$ lift to $E$. If this extension is non-trivial, the map 
$$
H^0(\oplus_{i=1}^{n-1} M_{j_i}) \otimes H^0(K_C \otimes M_1^*) \ra H^0(K_C \otimes M_1^* \otimes \oplus_{i=1}^{n-1} M_{j_i})
$$
is non-surjective. Since $K_C \otimes M_1^* \simeq M_2$, this contradicts Lemma \ref{lem3.1}. So the extension splits.

If there is no non-zero homomorphism $M_1 \ra E$, then by \eqref{eq3.1}, 
$$
h^0(E^* \otimes M_2) \geq n.
$$
If $E^* \otimes M_2 \simeq \cO_C^{\oplus n}$, the theorem follows. Otherwise, we have a map $\cO_C^{\oplus n} \ra E^* \otimes M_2$ which is not an isomorphism. 
Hence $E^* \otimes M_2$ has a section with a zero. So we have a non-zero homomorphism $\cO_C(p) \ra E^* \otimes M_2$ for some $p \in C$.
Dualizing we get $E \ra M_2(-p)$. Since $h^0(M_2(-p)) = 1$ and $E$ is generated by \cite[Theorem 2.4]{ln1}, this must factor as $E \ra \cO_C \ra M_2(-p)$.
But $h^0(E^*) = 0$, since $E$ is semistable. This is a contradiction.
\end{proof}

\begin{lem} \label{lem3.3}
 Let $C$ be a curve of type {\em (ii)}. Then there exists a non-trivial extension
$$
0 \ra M \ra F \ra M \ra 0
$$
with $h^0(F) = 4$, unique up to a scalar multiple.
\end{lem}

\begin{proof}
We need to show that the cokernel of the map
$$
H^0(M) \otimes H^0(K_C \otimes M^*) \ra H^0(K_C)
$$ 
is of dimension 1. The kernel of this map is $H^0(K_C \otimes M^{*2}) = H^0(\cO_C)$
and so of dimension 1. Hence also the cokernel is of dimension 1.
\end{proof}

\begin{theorem} \label{thm3.4}
Let $C$ be a curve of type {\em (ii)} and $E$ a bundle computing $\gamma_n(C)$. Then $E$ is a multiple extension of copies of $M$.
In particular $E$ computes $\Cl_n(C)$.  
\end{theorem}

\begin{proof}
For $n=1$ this is obvious. So suppose $n \geq 2$ and it is proved for bundles of rank $n-1$. In the same way as in the proof of the previous theorem we see 
that either there exists an injective homomorphism $M \ra E$ and $E$ computes $\Cl_n(C)$ or $h^0(E^* \otimes M) \geq n$. The argument is completed as in the proof of Theorem 
\ref{thm3.2}.
\end{proof}

\section{Trigonal curves of genus $\geq 5$}\label{g5}

Let $C$ be a trigonal curve of genus $g \geq 5$.

\begin{lem} \label{lem1}
Suppose $E$ is a semistable bundle of rank $n$ on $C$ with $d_E \leq n(g-1)$. Let $L$ be a line subbundle of $E$. Then every subbundle $F$ of $E/L$ has
$$
\mu(F) \leq 2g -2 - d_L.
$$
\end{lem}

\begin{proof}
Since $E$ is semistable, $d_L \leq g-1$. Suppose $F$ has rank $r$. The pullback of $F$ to $E$ is a subbundle of rank $r+1$ and degree $d_F + d_L$. By semistability of $E$,
$$
\frac{d_F + d_L}{r+1} \leq \frac{d_E}{n}.
$$ 
So $d_F \leq (r+1) \frac{d_E}{n} -d_L$ and
\begin{eqnarray*}
\mu(F) & \leq & \left( 1 + \frac{1}{r} \right) \frac{d_E}{n}  - \frac{d_L}{r}\\
&=& \frac{d_E}{n} + \frac{1}{r} \left( \frac{d_E}{n} - d_L \right)\\
&\leq & g-1 + \frac{1}{r} (g-1-d_L)\\
& \leq & 2g-2-d_L.
\end{eqnarray*}
\end{proof}

\begin{lem} \label{lem2}
Suppose $E$ computes $\gamma_n(C)$. Then there exists a non-zero homomorphism $T \ra E$. In particular, $d_E\ge3n$ and $E$ computes $\Cl_n(C)$.
\end{lem}

\begin{proof}
If there exists no non-zero homomorphism $T\to E$, then $h^0(E^* \otimes K_C \otimes T^*) \geq n(g-3)$ by Lemma \ref{lem2.1}. So, for 
all $p_1, \dots, p_{g-4} \in C$,
$$
h^0(E^* \otimes K_C \otimes T^*(-p_1 - \cdots -p_{g-4})) \geq n.
$$ 
If $E^* \otimes K_C \otimes T^*(-p_1 - \cdots - p_{g-4}) \simeq \cO_C^{\oplus n}$, then 
$$
E \simeq (K_C \otimes T^*(-p_1 - \cdots - p_{g-4}))^{\oplus n}.
$$
For general $p_1, \dots, p_{g-4}, \; h^0(K_C \otimes T^*(-p_1 - \cdots - p_{g-4})) = 2$, so
$\gamma(E) = g - 1 - 2 > 1$, a contradiction.
Hence there is a $q \in C$ such that there exists a non-zero homomorphism
$$
E \ra K_C \otimes T^*(-p_1 - \cdots - p_{g-4} -q).
$$
Since $h^0(E^*) = 0$ and $E$ is generated, 
$$
h^0(K_C \otimes T^*(-p_1 - \cdots - p_{g-4} -q)) \geq 2 \quad \mbox{and hence} \quad = 2.
$$
It follows that for general $r \in C$ we have 
$$
h^0(K_C \otimes T^*(-p_1- \cdots - p_{g-4} -q + r)) = 2.
$$ 
By \cite[III, Exercise B-5]{acgh} either $K_C \otimes T^*(-p_1 - \cdots - p_{g-4} -q + r)$  or its Serre dual contains $T$. 

If $K_C \otimes T^*(-p_1 - \cdots - p_{g-4} -q + r) \supset T$, then
$$
K_C \otimes T^{*2} = \cO_C (p_1 + \cdots +  p_{g-4} + q - r + p_1' + \cdots + p_{g-4}')
$$
for some points $p_1', \dots p_{g-4}'$ of $C$. So $K_C \otimes T^{*2}(r)$ has a section vanishing at arbitrary general points 
$p_1, \dots, p_{g-4}$ which implies 
$h^0(K_C \otimes T^{*2}(r)) \geq g-3$. However
$$
h^0(K_C \otimes T^{*2}(r)) = h^0(T^2(-r)) + 2g-7 -g+1 = g-4,
$$
since $h^0(T^2(-r)) = 2$ for general $r$.

Hence $T(p_1 + \cdots + p_{g-4} + q - r) \supset T$ and so
$$
h^0(\cO_C(p_1 + \cdots + p_{g-4}+ q - r)) \geq 1.
$$
Since $r$ is general and independent of $p_1, \dots, p_{g-4}$, this implies that
$$
h^0(\cO_C(p_1 + \cdots + p_{g-4} + q)) \geq 2 \quad \mbox{and hence} \quad = 2.
$$
So $\dim W^1_{g-3} \geq g-5$. But by Martens' theorem (see \cite[IV, Theorem 5.1]{acgh}), $\dim W^1_{g-3} \leq g-6$. 

The assertion that $d_E\ge3n$ follows from the semistability of $E$. But then the assumption that $\gamma(E)=1$ implies that  $h^0(E)\ge2n$.
\end{proof}

\begin{rem} {\em
For $g=5$ there is a much simpler proof. Here we have to prove that for general $p_1 \in C$ and arbitrary $q \in C$ we have
$$
h^0(K_C \otimes T^*(-p_1 - q)) \leq 1,
$$
which means that $K_C \otimes T^*(-p_1 -q) \not \simeq T$, i.e. $K_C \otimes T^{*2} \not \simeq \cO_C(p_1 + q)$.
But this is clearly true for general $p_1$.
} 
\end{rem}

\begin{rem}
{\em Note that there may exist bundles $E$ with $d_E<3n$ which contribute to $\gamma_n(C)$. In fact, taking into account \eqref{eq1.1}, it follows from \cite[Corollary 4.12]{ln} that the smallest degree for which such $E$ exists is $d_n$. Moreover $d_n\le3n$ and this inequality is strict if $n>\left[\frac{g}3\right]$. In this case, it follows from \cite[Theorem 4.15(a)]{ln} that $h^0(E)=n+1$, so $\gamma(E)=\frac{d_n-2}n>1$.}\end{rem}

\begin{lem} \label{lem3}
Let $C$ be a trigonal curve of genus $g \geq 5$ such that the map
\begin{equation} \label{eq4.1}
S^2H^0(K_C \otimes T^*) \ra H^0(K_C^2 \otimes T^{*2}) 
\end{equation}
is surjective. Then there exists no non-trivial extension
\begin{equation*}
0 \ra T \ra E \ra \oplus_{i=1}^rK_C \otimes T^* \ra 0
\end{equation*}
in which all sections of $\oplus_{i=1}^rK_C \otimes T^*$ lift to $E$.  
\end{lem}

\begin{rem} {\em
The surjectivity of the map \eqref{eq4.1} is equivalent to the normal generation of the line bundle $K_C \otimes T^*$. 
}
\end{rem}

\begin{proof}
It is sufficient to show that the map
$$
H^1((K_C\otimes T^*)^* \otimes T) \ra \Hom (H^0(K_C\otimes T^*),H^1(T))
$$
is injective or equivalently 
$$
H^0(K_C\otimes T^*) \otimes H^0(K_C \otimes T^*) \ra H^0(K_C^2\otimes T^{*2})
$$
is surjective. This holds by \eqref{eq4.1}.
\end{proof}

\begin{theorem} \label{thm4.4}
Let $C$ be a trigonal curve of genus $g \geq 5$ for which the map $S^2H^0(K_C \otimes T^*) \ra H^0(K_C^2 \otimes T^{*2})$ is surjective and $E$ a bundle computing $\gamma_n(C)$. Then
$$
E \simeq \oplus_{i=1}^ n T.
$$
\end{theorem}

\begin{proof}
For $n = 1$, the result is proved in Lemma \ref{lem2.2}. By induction we assume $n \geq 2$ and 
that the theorem is proved for bundles of rank $\leq n-1$. 

According to Lemma \ref{lem2}, there is a non-zero homomorphism $T \ra E$. Let $\widetilde T$ be the 
line subbundle generated by $T$. We apply Lemma \ref{lem1} with $L = \widetilde T$. 
Then every subbundle $F$ of $E/\widetilde T$ has $\mu(F) \leq 2g-5$. Moreover, every quotient bundle of $E/\widetilde T$ has slope $\geq \mu(E) \geq 3$ by semistability of $E$. So, if $Q$ is any quotient in the Harder-Narasimhan filtration of $E/\widetilde T$,
$$
3 \leq \mu(Q) \leq 2g-5.
$$ 

If neither $Q$ nor $K_C\otimes Q^*$ contributes to the appropriate Clifford index, either $h^0(Q) \leq \rk Q$ or $ h^1(Q) \leq \rk Q$. In either case $\gamma(Q) \geq 3$. On the other hand, if $Q$ or $K_C\otimes Q^*$ contributes to the Clifford index, $\gamma(Q) \geq 1$. So all quotients have $\gamma \ge 1$. 

Since $\gamma(E)=1$, it follows that $\gamma(\widetilde T)=1$ and so $\widetilde{T}=T$. Moreover, for every quotient $Q$, either $Q$ or $K_C\otimes Q^*$ computes the appropriate Clifford index 
and all sections lift at every stage. It follows by the inductive hypothesis and the definition of the Harder-Narasimhan filtration that we have an exact sequence
\begin{equation}\label{eq4.2}
0  \ra\oplus_{i=1}^r K_C\otimes T^*\ra E/T\ra\oplus_{i=1}^s T\ra0
\end{equation}
with $r,s\ge0$. If $r=0$ or $s=0$, the result follows from Lemma \ref{lem2.2} or Lemma \ref{lem3}. If $r,s\ge1$, then pulling back the subbundle in the sequence \eqref{eq4.2} contradicts the semistability of $E$. This completes the proof.
\end{proof}

\begin{cor}
If $C$ is trigonal of genus $g = 5$ or $6$ or $g \geq 17$ and $E$ is a bundle computing $\gamma_n(C)$, then
$$
E \simeq \oplus_{i=1}^n T.
$$ 
\end{cor}

\begin{proof}
We have to show that the map \eqref{eq4.1} is surjective. For $g \geq 17$ this follows from \cite[Theorem 2(a)]{gl} (see also \cite[Remark 5.2]{ln3}).

For $g=5$, both spaces are of dimension 6. So this is equivalent to saying that the map \eqref{eq4.1} is injective. If not, 
the image of $C$ in $\PP^2$ given by the sections of $K_C \otimes T^*$
is contained in a conic. Since $d_{K_C \otimes T^*} = 5$, this is impossible.

For $g = 6$, we have $h^0(K_C^2 \otimes T^{*2}) = 9$ 
and $\dim S^2H^0(K_C \otimes T^*) = 10$. We need to prove that the image of $C$ in $\PP^3$ given by $K_C \otimes T^*$ 
is not contained in 2 independent quadrics. This follows from the fact that the degree of the image is 7. 
\end{proof}

\begin{rem} {\em
To obtain the result for $g=7$, we would have to prove that the image of $C$ in $\PP^4$ given by $K_C \otimes T^*$ is contained in at most 3 independent quadrics.

More generally, note that
$$h^0(K_C\otimes T^*)=h^0(T)+2g-2-3-g+1=g-2,$$
while
$$h^0(K_C^2\otimes T^{*2})=4g-4-6-g+1=3g-9.$$
So the difference in dimension between the two sides of \eqref{eq4.1} is easily seen to be $\frac12(g-4)(g-5)$. This implies that, if \eqref{eq4.1} fails to be surjective, then the image of $C$ under the morphism given by the sections of the very ample line bundle $K_C\otimes T^*$ is contained in a quadric of corank $g-5$, hence of rank $3$. This alone, however, is not sufficient to prove non-surjectivity.}\end{rem} 

\section{Genus 3}\label{g3}

Let $C$ be a non-hyperelliptic curve of genus $3$. Then $C$ is trigonal and (see \eqref{eq1.1})
\begin{equation}\label{eq5.1}
d_1=3,\ \ d_2=4,\ \ d_r=r+3 \text{ for }r\ge3.
\end{equation}
The Clifford index $\Cl_n(C)$ is not defined for any $n$. In fact, if $E$ is a semistable bundle of rank $n$ with $h^0(E)\ge2n$, then Clifford's theorem for semistable bundles implies that $d_E\ge2n$ with equality possible only on a hyperelliptic curve. So in our case $d_E>2n$, contradicting the assumption that $d_E\le n(g-1)$.

The situation for $\gamma_n(E)$ is quite different. To see this, we begin with a definition and a lemma.

For any generated vector bundle, we define the {\it dual span bundle} $D(E)$ by the exact evaluation sequence
$$0\ra D(E)^*\ra H^0(E)\otimes{\mathcal O}_C\ra E\ra 0.$$

\begin{lem}\label{lem5.1}
Let $C$ be a trigonal curve of genus $3$ and let $L$ be a line bundle on $C$ of degree $d_n$ with $h^0(L)=n+1$. Then $L$ is generated and $D(L)$ is semistable with $h^0(D(L))=n+1$.
\end{lem}
\begin{proof} See the proofs of \cite[Proposition 4.9(d) and Theorem 4.15(a)]{ln}. (Although there is a blanket assumption in \cite{ln} that $g\ge4$, this is not used in these proofs.)
\end{proof}

\begin{prop}\label{prop5.2}
Let $C$ be a trigonal curve of genus $3$ and $n\ge2$. Then $\gamma_n(C)$ is defined and
\begin{equation}\label{eq5.2}
\gamma_n(C) =\left\{\begin{array}{cl}1&\text{ if }n\text{ is even }\\1+\frac1n&\text{ if }n\text{ is odd }.\end{array}\right.
\end{equation}\end{prop}
\begin{proof} Let $L$ be a line bundle of degree $d_n$ with $h^0(L)=n+1$. By Lemma \ref{lem5.1}, $D(L)$ is semistable with $h^0(D(L))=n+1$. Moreover $\mu(D(L))=\frac{d_n}n\le2$ and 
\begin{equation}\label{eq5.3}
\gamma(D(L)) =\left\{\begin{array}{cl}1&\text{ if }n=2\\1+\frac1n&\text{ if }n\ge3.\end{array}\right.
\end{equation}
Noting that $D(K_C)^{\oplus n/2}$ is semistable when $n$ is even, this shows that the right hand side of \eqref{eq5.2} gives an upper bound for $\gamma_n(C)$. It remains to prove that this is also a lower bound for $\gamma(E)$, where $E$ is a bundle contributing to $\gamma_n(C)$.

Suppose $E$ is such a bundle with $h^0(E)=n+s$, $s\ge1$. By definition, we have $0<d_E\le2n$. If $d_E<2n$, then, by \cite{bgn, m1},  $d_E-n\ge3s$, so 
\begin{equation}\label{eq5.4}
\gamma(E)=\frac1n(d_E-2s)\ge1+\frac{s}n\ge1+\frac1n.
\end{equation}
If $d_E=2n$ and $E$ is stable, then, by \cite{m2}, \eqref{eq5.4} still holds unless $E\simeq D(K_C)$, in which case $\gamma(E)=1$. (For the stability of $D(K_C)$, see \cite[Corollary 3.5]{pr}.)

Suppose now that $E$ is any bundle of degree $2n$ contributing to $\gamma_n(C)$. We consider a Jordan-H\"older filtration of $E$. We have shown that every quotient $Q$ in this filtration has $\gamma(Q)\ge1$, so $\gamma(E)\ge1$. If $n$ is odd, there must exist a quotient $Q_0$ of odd rank $r$, which therefore has $\gamma(Q_0)\ge1+\frac1r$ by \eqref{eq5.4}. So, in this case, $\gamma(E)>1$ and hence $\gamma(E)\ge1+\frac1n$.  \end{proof}

\begin{theorem}\label{thm5.3} Let $C$ be a trigonal curve of genus $3$ and $n\ge2$. Then
\begin{itemize}
\item[(i)] if $n$ is even, the only bundle computing $\gamma_n(C)$ is $D(K_C)^{\oplus n/2}$;
\item[(ii)] if $n$ is odd, the bundles $E$ computing $\gamma_n(C)$ with $d_E<2n$ have the form $D(L)$, where $L$ is a line bundle of degree $d_n$ with $h^0(L)=n+1$;
\item[(iii)] if $n$ is odd, the associated graded object of any bundle computing $\gamma_n(C)$ with $d_E=2n$ has one of the following forms:
\item $D(K_C)^{\oplus (n-1)/2}\oplus {\mathcal O}_C(p,q)$ for some $p,q\in C$
\item $D(K_C)^{\oplus (n-3)/2}\oplus D(M)$, where $M$ is a line bundle of degree $6$, $M\not\simeq K_C(p,q)$.
\end{itemize}\end{theorem}
\begin{proof} (i) It follows from Proposition \ref{prop5.2} and its proof that the associated graded object of any bundle computing $\gamma_n(C)$ for $n$ even has the form $D(K_C)^{\oplus n/2}$. The result will follow if we show that any extension 
$$0\to D(K_C)\ra E\ra D(K_C)\ra 0,$$
for which all sections of the quotient lift to $E$, splits. This is equivalently to showing that
\begin{equation}\label{eq5.5}
H^0(D(K_C))\otimes H^0(K_C\otimes D(K_C)^*)\ra H^0(K_C\otimes D(K_C)\otimes D(K_C)^*)
\end{equation}
is surjective. Using the sequence
$$0\ra K_C^*\ra H^0(D(K_C))\otimes{\mathcal O_C}\ra D(K_C)\ra0$$
and the fact that $D(K_C)$ is stable, we see that the kernel of \eqref{eq5.5} is $H^0(D(K_C)^*)=0$. Now $h^0(D(K_C))=3$ and $$h^0(K_C\otimes D(K_C)^*)=h^1(D(K_C))=3$$ by Riemann-Roch. Moreover 
\begin{eqnarray*}
h^0(K_C\otimes D(K_C)\otimes D(K_C)^*)&=&h^1(D(K_C)\otimes D(K_C)^*)\\&=&h^0(D(K_C)\otimes D(K_C)^*)+4g-4=9
\end{eqnarray*} 
by Riemann-Roch and the stability of $D(K_C)$. It follows that \eqref{eq5.5} is an isomorphism.

(ii) For equality in \eqref{eq5.4}, we must have $s=1$ and $d_E=n+3$ and then $d_E=d_n$ by \eqref{eq5.1}. It is now easy to see that $E\simeq D(L)$ for some $L$ of degree $d_n$ with $h^0(L)=n+1$.

(iii) In order to compute $\gamma_n(C)$, a bundle $E$ of slope $2$ must have precisely one quotient $Q$ in a Jordan-H\"older filtration of rank $r$ with $\gamma(Q)=1+\frac1r$.  By the proof of  Proposition \ref{prop5.2}, it follows that $Q\simeq{\mathcal O}_C(p,q)$ or $Q\simeq D(M)$ for some $M$ of degree $6$. In the latter case, $D(M)$ is stable unless $M\simeq K_C(p,q)$ for some $p,q$ (see \cite[Theorem 1.2]{b}). 
\end{proof}

\begin{rem}{\em In case (iii) of the theorem, we can be more precise. For example, the only bundles computing $\gamma_3(C)$ are 
\begin{itemize}
\item $D(K_C)\oplus {\mathcal O}_C(p,q)$;
\item for each $p,q$, a unique non-trivial extension
$$0\ra D(K_C)\ra E\ra {\mathcal O}_C(p,q) \ra0;$$
\item $D(M)$ with $M$ a line bundle of degree $6$.
\end{itemize}
In the third case, $D(M)$ is stable unless $M\simeq K_C(p,q)$; moreover we have non-trivial exact sequences
$$0\ra {\mathcal O}_C(p,q)\ra D(K_C(p,q))\ra D(K_C)\ra0.$$
}\end{rem}

\end{document}